\documentclass[reqno,12pt]{amsart}
\def\vint_#1{\mathchoice%
          {\mathop{\kern 0.2em\vrule width 0.6em height 0.69678ex depth -0.58065ex
                  \kern -0.8em \intop}\nolimits_{\kern -0.4em#1}}%
          {\mathop{\kern 0.1em\vrule width 0.5em height 0.69678ex depth -0.60387ex
                  \kern -0.6em \intop}\nolimits_{#1}}%
          {\mathop{\kern 0.1em\vrule width 0.5em height 0.69678ex
              depth -0.60387ex
                  \kern -0.6em \intop}\nolimits_{#1}}%
          {\mathop{\kern 0.1em\vrule width 0.5em height 0.69678ex depth -0.60387ex
                  \kern -0.6em \intop}\nolimits_{#1}}}
\def\vintslides_#1{\mathchoice%
          {\mathop{\kern 0.1em\vrule width 0.5em height 0.697ex depth -0.581ex
                  \kern -0.6em \intop}\nolimits_{\kern -0.4em#1}}%
          {\mathop{\kern 0.1em\vrule width 0.3em height 0.697ex depth -0.604ex
                  \kern -0.4em \intop}\nolimits_{#1}}%
          {\mathop{\kern 0.1em\vrule width 0.3em height 0.697ex de pth -0.604ex
                  \kern -0.4em \intop}\nolimits_{#1}}%
          {\mathop{\kern 0.1em\vrule width 0.3em height 0.697ex depth -0.604ex
                  \kern -0.4em \intop}\nolimits_{#1}}}

\usepackage{a4wide}
\usepackage{mathtools}
\usepackage{amsthm}
\usepackage{amssymb}
\usepackage{hyperref}
\usepackage{color}
\mathtoolsset{showonlyrefs}
\usepackage[obeyFinal]{todonotes}
\usepackage{csquotes}
\usepackage{graphicx}
\usepackage{epsfig,color}
\usepackage{enumitem}

\numberwithin{equation}{section}
\newtheorem{theorem}{Theorem}[section]
\newtheorem{lemma}[theorem]{Lemma}

\theoremstyle{definition}
\newtheorem{definition}[theorem]{Definition}

\theoremstyle{remark}

\renewcommand{\d}{{\mathrm d}}
\newcommand{\R}{\mathbb{R}}
\renewcommand{\S}{\mathbb{S}}
\newcommand{\N}{\mathbb{N}}

\newcommand{\F}{\mathcal{F}}
\newcommand{\B}{\mathcal{B}}

\newcommand{\spt}{\mathrm{spt}}
\newcommand{\Reg}{\mathrm{Reg}}
\newcommand{\Sing}{\mathrm{Sing}}
\newcommand{\opt}{\mathrm{Opt}}
\newcommand{\id}{\mathrm{id}}
\newcommand{\geo}{\mathrm{Geo}}

\begin{document}

\title[Optimal transport maps on Alexandrov spaces revisited]
{Optimal transport maps on Alexandrov spaces revisited}

\author{Tapio Rajala}
\author{Timo Schultz}

\address{University of Jyvaskyla\\
         Department of Mathematics and Statistics \\
         P.O. Box 35 (MaD) \\
         FI-40014 University of Jyvaskyla \\
         Finland}
\email{tapio.m.rajala@jyu.fi}
\email{timo.m.schultz@student.jyu.fi}

\thanks{Both authors partially supported by the Academy of Finland.}
\date{\today}

%%%%%%%%%%%%%%%%%%%%%%%%%%%%%%%%%%%%%%%%%%%%%%%%%%%%%%%%%%%%%%%%%%%%%

\begin{abstract}
We give an alternative proof for the fact that in $n$-dimensional Alexandrov spaces with curvature bounded below there exists a unique optimal transport plan from any purely $(n-1)$-unrectifiable starting measure, and that this plan is induced by an optimal map.
\end{abstract}

\maketitle

\section{Introduction}

%{\color{red}Do we want to use $\mu_0$ and $\mu_1$, or $\mu$ and $\nu$? Pick one and change below.}

The problem of optimal mass transportation has a long history, starting from the work of Monge \cite{Monge} in the late 18th century. In the original formulation of the problem, nowadays called the Monge-formulation, the problem is to find the transport map $T$ minimizing the transportation cost
\begin{equation}\label{eq:Monge}
\int_{\R^n}c(x,T(x))\,\d\mu_0(x),
\end{equation}
among all Borel maps $T \colon \R^n \to \R^n$ transporting a given probability measure $\mu_0$ to another given probability measure $\mu_1$, that is, $T_\sharp\mu_0 = \mu_1$. In the original problem of Monge, the cost function $c(x,y)$ was the Euclidean distance. Later, other cost functions have been considered, in particular much of the study has involved the distance squared cost, $c(x,y) = |x-y|^2$, which is the cost studied also in this paper.

In the Monge-formulation \eqref{eq:Monge} of the optimal mass transportation problem the class of admissible maps $T$ that send $\mu_0$ to $\mu_1$ is in most cases not closed in any suitable topology. To overcome this problem, Kantorovich \cite{Kantorovitch1,Kantorovitch2} considered a larger class of optimal transports, namely, measures $\pi$ on $\R^n \times \R^n$ such that the first marginal of $\pi$ is $\mu_0$ and the second is $\mu_1$. Such measures $\pi$ are called transport plans. Kantorovich's relaxation leads to the so-called Kantorovich-formulation of the problem,
\begin{equation}\label{eq:Kantorovich}
\inf_{\pi}\int_{\R^n\times \R^n}c(x,y)\,\d\pi(x,y).
\end{equation}
Due to the closedness of the admissible transport plans and the lower semi-continuity of the cost, minimizers exist in the Kantorovich-formulation under very mild assumptions on the underlying space and the cost $c$. %{\color{red}reference?}.

For the quadratic cost in the Euclidean space, it was shown independently by Brenier \cite{Brenier}
and Smith and Knott \cite{SmithKnott} that having $\mu_0$ absolutely continuous with respect to the Lebesgue measure guarantees that the optimal transport plans (minimizer of \eqref{eq:Kantorovich}) are unique and given by a transport map. Moreover, the optimal transport map is given by a gradient of a convex function.

The results of Brenier and of Smith and Knott have been generalized in many ways. The most important directions of generalization have been: going from the underlying space $\R^n$ to other metric spaces, considering other cost functions, and relaxing the assumption of the starting measure being absolutely continuous with respect to the reference measure (here the Lebesgue measure). In this paper, we study the direction of relaxing the absolute continuity in a more general metric space setting, the Alexandrov spaces. 
We note that one should be able to generalize our proof for more general costs, such as the distance to a power $p\in (1,\infty)$. In order to keep the presentation simpler, we concentrate here on the distance squared cost.

The existence of optimal transportation maps in Alexandrov spaces with curvature bounded below for starting measures that are absolutely continuous with respect to the reference Hausdorff measure was proven by Bertrand \cite{Bertrand}. 
Later Bertrand improved this result \cite{Bertrand_habi} by relaxing the assumption on the starting measure to give zero measure to $c-c$-hypersurfaces.
Here we provide an alternative proof for the result of Bertrand under the slightly stronger assumption on the starting measure of pure $(n-1)$-unrectifiability (see Definition \ref{def:unrectifiability} for the definition of purely $(n-1)$-unrectifiability).

\begin{theorem}\label{thm:main}
 Let $(X,d)$ be an $n$-dimensional Alexandrov space with curvature bounded below. Then for any pair of measures $\mu_0,\mu_1 \in \mathcal P_2(X)$ such that $\mu_0$ is purely $(n-1)$-unrectifiable, there exists a unique optimal transport plan from $\mu_0$ to $\mu_1$ and this transport plan is induced by a map.
\end{theorem}

The contribution of this paper is to provide a different approach to showing the existence and uniqueness of optimal transport maps than what was used by Bertrand in \cite{Bertrand,Bertrand_habi}. In \cite{Bertrand}, Bertrand used the local $(1+\varepsilon)$-biLipschitz maps to $\R^n$ on the regular set of $X$, and the general existence of Kantorovich potentials and their Lipschitzness. Since the singular set of $X$ is at most $(n-1)$-dimensional, and the Rademacher's theorem on $\R^n$ can be restated in $X$ via the biLipschitz maps, Bertrand concluded that the optimal transport is concentrated on a graph that is given by applying the exponential map to the gradient of the Kantorovich potential. 
In \cite{Bertrand_habi}, Bertrand considered the problem in boundaryless Alexandrov spaces. He used Perelman's DC calculus to translate the problem to differentiability of convex functions on Euclidean spaces. Then the result follows from the characterization of nondifferentiability points of convex functions due to Zaj\'\i \v cek \cite{zajicek}.

In this paper, 
%we first recall from the proof by Otsu and Shioya \cite{OtsuShioya} that the singular set in an $n$-dimensional Alexandrov space $X$ is in fact $(n-1)$-rectifiable (Theorem \ref{thm:singsize}). Then
we translate a contradiction argument (Lemma \ref{lma:cyclicalcontradicition}) from the Euclidean space (which uses just cyclical monotonicity in certain geometric configurations) to the space $X$ via the $(1+\varepsilon)$-biLipschitz charts. In order to use the contradiction argument, we need to get all the used distances to be comparable. For this we use the fact that the directions of geodesics are well-defined in the biLipschitz charts (Theorem \ref{thm:epsilonneighborhood}) and thus we can contract along the geodesics without changing the geometric configuration too much. Finally, the geometric configurations that result in the contradiction via cyclical monotonicity are given by the pure $(n-1)$-unrectifiability (Lemma \ref{lma:unrecteverywhere}).

Let us comment on the history of the sufficient assumptions on $\mu_0$.
The assumption of pure $(n-1)$-unrectifiability was shown by McCann \cite{McCann} to be sufficient for the existence of optimal maps in the case of Riemannian manifolds. A sharper condition based on the characterization by Zaj\'\i \v cek \cite{zajicek}
of the set of nondifferentiability points of convex functions 
was first used in the Euclidean context by Gangbo and McCann \cite{GangboMcCann} when they showed that having an initial measure that gives zero mass to $c-c$ -hypersurfaces is sufficient to give the existence of optimal maps. It was then shown by Gigli \cite{Gigli2} that even in the Riemannian manifold context the sharp requirement for the starting measure to have optimal maps for any target measure is indeed that it gives zero measure to $c-c$ -hypersurfaces. %The result of Gigli relies on 
It still remains open whether zero measure of $c-c$ -hypersurfaces also gives a full characterization in the case of Alexandrov spaces. One of the directions, the sufficiency, was obtained by Bertrand \cite{Bertrand_habi}.

The existence of optimal maps has been studied in wider classes of metric measure spaces that satisfy some form of Ricci curvature lower bounds or weak versions of measure contraction property. 
These classes include $CD(K,N)$-spaces that were introduced by Lott and Villani \cite{LottVillani}, and by Sturm \cite{SturmI,SturmII}, $MCP(K,N)$-spaces (see Ohta \cite{Ohta}), and $RCD(K,N)$ spaces that were first introduced by  Ambrosio, Gigli and Savar\'e \cite{AmbrosioGigliSavare2} (see also the improvements and later work by Ambrosio, Gigli, Mondino and Rajala \cite{AmbrosioGigliMondinoRajala}, Erbar, Kuwada and Sturm \cite{ErbarKuwadaSturm} and Ambrosio, Mondino and Savar\'e \cite{AmbrosioMondinoSavare}). All of these classes contain Alexandrov spaces with curvature lower bounds, see Petrunin \cite{Petrunin}.

It was first shown by Gigli \cite{Gigli}, that in nonbranching $CD(K,N)$-spaces you do have the existence of optimal maps provided that the starting measure is absolutely continuous with respect to the reference measure. In all the subsequent work, the assumption has been the same for the starting measure, and it would be interesting to see if it can be relaxed also in the more general context of metric measure spaces with Ricci curvature lower bounds.

Also a metric version of Brenier's theorem was studied by Ambrosio, Gigli and Savar\'e \cite{AmbrosioGigliSavare1}. They did not obtain the existence of optimal maps, but showed that at least the transportation distance is given by the Kantorovich potential. Later, Ambrosio and Rajala \cite{AmbrosioRajala} showed that under sufficiently strong nonbranching assumptions one can conclude the existence of optimal maps.

Rajala and Sturm \cite{RajalaSturm} noticed that strong $CD(K,\infty)$ spaces, and hence $RCD(K,\infty)$ spaces are at least essentially nonbranching, and that this weaker form of nonbranching is sufficient for carrying out Gigli's proof. This result was later improved by Gigli, Rajala and Sturm \cite{GigliRajalaSturm}. Essential nonbranching was then studied together with the measure contraction property $MCP(K,N)$ by Cavalletti and Huesmann \cite{CavallettiHuesmann} and Cavalletti and Mondino \cite{CavallettiMondino}, and finally it was shown by Kell \cite{Kell} that under a weak type measure contraction property, the essential nonbranching characterizes the uniqueness of optimal transports and that the unique optimal transport is given by a map for absolutely continuous starting measures.

The existence of optimal transport maps in $CD(K,N)$ spaces without any extra assumption on nonbranching is still an open problem. An intermediate definition between $CD(K,N)$ and essentially nonbranching $CD(K,N)$, called very strict $CD(K,N)$, was studied by Schultz \cite{Schultz}. He showed that in these spaces one still has optimal transport maps even if the space could be highly branching and the optimal plans non-unique. It is also worth noting that if one drops the assumption of essential nonbranching for $MCP(K,N)$, then optimal transport maps need not exist. This is seen from the examples by Ketterer and Rajala \cite{KettererRajala}.
\medskip

% {\color{red}Recall other results in similar directions. Existence of maps: Ambrosio and Rigot \cite{AmbrosioRigot},  Cavalletti \cite{Cavalletti}\\
% }
The paper is organized as follows. In Section \ref{sec:preli} we recall basic things about rectifiability, Alexandrov spaces and optimal mass transportation. While doing this, we also present a few facts that easily follow from well-known results: purely $n-1$-unrectifiable measures have mass in all directions (Lemma \ref{lma:unrecteverywhere}), the singular set in an Alexandrov space is $(n-1)$-rectifiable (Theorem \ref{thm:singsize}), gradients of geodesics exist in charts in Alexandrov spaces (Theorem \ref{thm:epsilonneighborhood}) and the failure of cyclical monotonicity persists after small perturbations (Lemma \ref{lma:cyclicalcontradicition}). In Section \ref{sec:proof} we then put these things together and prove Theorem \ref{thm:main}.

\section{Preliminaries}\label{sec:preli}
In this paper $(X,d)$ always refers to a complete and locally compact length space.
By a length space we mean a metric space  where the distance between any two points $x$ and $y$ is equal to the infimum of lengths  of curves connecting $x$ and $y$. 
By the Hopf-Rinow-Cohn-Vossen Theorem, our spaces $(X,d)$ are then geodesic, proper and, in particular, separable. A space is called geodesic, if any two points in the space can be connected by a geodesic.
By a geodesic we mean a constant speed length minimizing curve $\gamma\colon [0,1]\to X$. 
Notice that we parametrize all the geodesics by the unit interval.
We denote the space of geodesics of $X$ by $\geo(X)$ and equip it with the supremum-distance.
By a (geodesic) triangle $\Delta (x,y,z)$ we mean points $x,y,z\in X$ and any choice of geodesics $[x,y]$, $[y,z]$ and $[x,z]$ pairwise connecting them.

\subsection{Rectifiability}

For our Theorem \ref{thm:main} the starting measure $\mu_0$ is diffused enough if it is purely $n-1$-unrectifiable. Let us recall this notion.

\begin{definition}\label{def:unrectifiability}
A set $A \subset X$ is called \emph{(countably) $k$-rectifiable} if there exist Lipschitz maps $f_i \colon E_i \to X$ from Borel sets $E_i \subset \R^k$ for $i \in \N$, such that $A \subset \bigcup_{i\in\N}f_i(E_i)$.

A measure $\mu$ is called \emph{purely $k$-unrectifiable}, if $\mu(A) = 0$ for every $k$-rectifiable set $A$.
\end{definition}

The property of purely unrectifiable measures that we use is that they have mass in all directions. This is made precise using (one-sided) cones that are defined as follows.
Given $x \in \R^n$, $\theta \in \S^{n-1}$, $\alpha > 0$ and $r> 0$, we denote the open cone at $x$ in direction $\theta$ with opening angle $\alpha$, by
\[
C(x,\theta,\alpha) := \left\{y \in \R^n\,:\, \langle y-x, \theta\rangle > \cos(\alpha)|y-x|\right\}.
\]

\begin{lemma}\label{lma:unrecteverywhere}
 Let $\mu$ be a purely $(n-1)$-unrectifiable measure on $\R^n$ and let $E \subset \R^n$ with $\mu(E) > 0$. Then at $\mu$-almost every $x \in E$ we have $C(x,\theta, \alpha) \cap B(x,r) \cap E \ne \emptyset$
 for all $\theta \in \S^{n-1}$, $\alpha > 0$ and $r> 0$.
\end{lemma}
\begin{proof}
 Suppose that there is a subset $E_0 \subset E$ with $\mu(E_0)>0$ such that the conclusion fails, i.e. for every $x \in E_0$ there exist $\theta_x \in \S^{n-1}$, $\alpha_x > 0$ and $r_x> 0$ such that $C(x,\theta_x, \alpha_x) \cap B(x,r_x) \cap E = \emptyset$. Since
 \[
 C(x,\theta, \alpha) \cap B(x,r) \subset C(x,\theta, \alpha') \cap B(x,r')
 \]
 if $\alpha' \ge \alpha$ and $r'\ge r$, there exist $r>0$ and $\alpha >0$ such that the subset
 \[
 \{x \in E_0\,:\,C(x,\theta_x, \alpha) \cap B(x,r) \cap E = \emptyset\}
 \]
 has positive $\mu$-measure. By considering a countable dense set of directions $\{\theta_i\}_{i \in \N}$, we have that there exists one fixed direction $\theta_i$ such that the set
 \[
 E_1 := \{x \in E_0\,:\,C(x,\theta_i, \alpha/2) \cap B(x,r) \cap E = \emptyset\}
 \]
 has positive $\mu$-measure. But now, for evey $x \in \R^n$, the set $E_1\cap B(x,r/2)$ is contained in a Lipschitz graph and hence $E_1$ is an $(n-1)$-rectifiable set, giving a contradiction with the pure $(n-1)$-unrectifiability of $\mu$.
\end{proof}

\subsection{Alexandrov spaces}

Let us recall some basics about Alexandrov spaces.
Unless we provide another source, all the following definitions and results can be found in \cite{BuragoBuragoIvanov}.

Alexandrov spaces generalize sectional curvature bounds by means of comparison to constant curvature model spaces. Alexandrov spaces can be defined for instance by comparing geodesic triangles of a metric space to the corresponding ones in a model space. Let us next give precise definitions. 

For each $k\in\R$, let $M_k$ be a simply connected surface with constant sectional curvature equal to $k$, that is, for negative $k$, $M_k$ is a scaled hyperbolic plane, for $k=0$, $M_k$ is the Euclidean plane, and for positive $k$, $M_k$ is a (round) sphere. Let us denote the distance between two points $x,y \in M_k$ by $|x-y|$.

Let $k\in\R$. For a triplet $x,y,z\in X$, let $\tilde x,\tilde y,\tilde z\in M_k$ be points so that the triangles $\Delta (x,y,z)$ and $\Delta (\tilde x,\tilde y,\tilde z)$ have the same side lengths, that is, $d(x,y)=|\tilde x-\tilde y|,d(y,z)=|\tilde y-\tilde z|,d(x,z)=|\tilde x-\tilde z|$. We call the triangle $\Delta (\tilde x,\tilde y,\tilde z)$  a comparison triangle for $\Delta (x,y,z)$.
For a triangle $\Delta(x,y,z)$ in $X$ we denote by $\tilde\measuredangle_k(y,x,z)$ the comparison angle at $\tilde x$ in the comparison triangle $\Delta(\tilde x,\tilde y,\tilde z)$ in $M_k$.

\begin{definition}[Alexandrov space] We say that $(X,d)$ is an Alexandrov space (with curvature bounded below by $k$) if there exists $k\in \R$ so that for each point $p\in X$ there exists a neighbourhood $U$ of $p$ for which the following holds. If $\Delta (x,y,z)\subset U$, $\Delta (\tilde x,\tilde y,\tilde z)$ its comparison triangle in $M_k$, and $w\in [x,y]$, $\tilde w\in [\tilde x,\tilde y]$ with $d(x,w)=|\tilde x-\tilde w|$, then $d(w,z)\ge |\tilde w-\tilde z|$. \end{definition}
 
An Alexandrov space might have infinite (Hausdorff) dimension. In this paper we study only finite dimensional Alexandrov spaces. Recall that in an Alexandrov space every open nonempty set has the same dimension, so the dimension of an Alexandrov space is always well defined.
Moreover, the dimension is either an integer or infinity.
From now on, the space $(X,d)$ is assumed to be an $n$-dimensional Alexandrov space with curvature bounded below by $k \in \R$ with $n \in \N$.

% for any ball $B$ in an Alexandrov space $X$, the Hausdorff dimension $\dim_{\H}B$ and the topological dimension $\dim_\tau B$ coincide with the Hausdorff dimension $\dim_{\H}X$ of the whole space $X$. Moreover, this dimension is either an integer or infinity. Therefore, it is natural to regard this dimension as the dimension of the Alexandrov space $X$. We will denote it as $\dim X$.   

We will use the fact that our purely ($n-1$)-unrectifiable starting measures $\mu_0$ live on the regular set of the space, that has nice charts. Let us recall the notion of regular and singular points.

\begin{definition}
 A point $p \in X$ is called \emph{regular}, if the space of directions $\Sigma_p$ at $p$ is isometric to the standard sphere $\S^{n-1}$, or equivalently, if the Gromov-Hausdorff tangent at $p$ is the Euclidean $\R^n$.
 A point $p \in X$ that is not regular is called \emph{singular}. The set of regular points of $X$ is denoted by 
 $\Reg(X)$ and the set of singular points by $\Sing(X)$.
\end{definition}

The following  result is from \cite{OtsuShioya} (see also \cite{BuragoGromovPerelman}). It implies that our starting measures $\mu_0$ give zero measure to the singular set.
%{\color{red}Check what is actually done in \cite{BuragoGromovPerelman}.}

\begin{theorem}\label{thm:singsize}
 The set $\Sing(X)$ is $(n-1)$-rectifiable.
\end{theorem}
\begin{proof}
Notice that \cite[Theorem A]{OtsuShioya} states that $\Sing(X)$ has Hausdorff dimension at most $n-1$.
However, the proof easily gives the stronger conclusion of $(n-1)$-rectifiability. Namely, observe that in the proof of \cite[Theorem A]{OtsuShioya} Otsu and Shioya show that $\Sing(X)$ is contained in Lipschitz images from subsets of the spaces of directions $\Sigma_p$ for countably many points $p \in X$. Since the points $p$ are only needed to locally form a maximal $\varepsilon$-discrete net in $X$, they can be chosen to be regular points of $X$. Thus, $\Sing(X)$ is contained in countably many Lipschitz images from subsets of $\S^{n-1}$ and is therefore $(n-1)$-rectifiable.
\end{proof}

Let us then recall a well-known consequence of the nonbranching property of Alexandrov spaces. For its proof, we need the notion of an angle.
Let $\alpha, \beta \colon [0,1] \to X$ be two constant speed geodesics emanating from the same point $p = \alpha(0) = \beta(0)$. Let us denote by $\theta_k(t,s) := \tilde\measuredangle_k(\alpha(t),p,\beta(s))$ the angle at $\tilde p$ of the comparison triangle $\Delta (\tilde p, \tilde\alpha(t),\tilde\beta(s))$ in $M_k$ of $\Delta (p, \alpha(t),\beta(s))$.
In Alexandrov spaces the angle 
\begin{equation}
\measuredangle(\alpha,\beta) := \lim_{t,s \searrow 0}\theta_k(t,s)
\end{equation}
is well-defined for every pair of geodesics $\alpha, \beta$ emanating from the same point.
Moreover, by Alexandrov convexity (see for instance \cite[Section 2.2]{Shiohama}) the quantity $\theta_k(t,s)$ is  monotone non-increasing in both variables $t$ and $s$.

\begin{lemma}\label{lma:positive_distance_limit}
 Let $\gamma_1, \gamma_2 \colon [0,1] \to X$ be be two constant speed geodesics with $\gamma_1(0)=\gamma_2(0)$ and $\gamma_1(1) \ne \gamma_2(1)$. Then
 \[
 \lim_{t \searrow 0} \frac{d(\gamma_1(t),\gamma_2(t))}{t} > 0.
 \]
\end{lemma}
\begin{proof}
We may assume $\ell(\gamma_1) \ge \ell(\gamma_2)$. If $\ell(\gamma_1) > \ell(\gamma_2)$, then by triangle inequality
$d(\gamma_1(t),\gamma_2(t)) \ge t (\ell(\gamma_1) - \ell(\gamma_2))$, giving the claim. If $\ell(\gamma_1) = \ell(\gamma_2)$, then $\theta_k(1,1) = \tilde\measuredangle_k(\gamma_1(1),x,\gamma_2(1)) > 0$.
Then by Alexandrov convexity, $\measuredangle(\gamma_1,\gamma_2) \ge \theta_k(1,1) > 0$, 
and thus by the cosine law
\[
\frac{d(\gamma_1(t),\gamma_2(t))}{t} \to \ell(\gamma_1)\sqrt{2-2\cos(\measuredangle(\gamma_1,\gamma_2))} > 0,
\]
as $t \to 0$.
\end{proof}

%  Then we have the following result (see for instance \cite[Section 2.2]{Shiohama}).
% \begin{lemma}[Alexandrov convexity]\label{lma:alex_conv}
% Let $(X,d)$ have curvature bounded below by $k \in \R$. Then 
% $\theta_k(t,s)$ is monotone non-increasing in both variables for any two geodesics $\alpha$ and $\beta$ emanating from the same point.
% \end{lemma}

Our aim is to arrive at a contradiction with cyclical monotonicity at a small scale near a regular point. We will transfer the Euclidean argument to the Alexandrov space $X$ using the following standard charts $\varphi$. Since we need the existence of directions of geodesics in these charts, we write the existence down explicitly inside the following theorem.

\begin{theorem}\label{thm:epsilonneighborhood}
 For every $p \in \Reg(X)$ and every $\varepsilon > 0$ there exist a neighborhood $U$ of $p$ and a $(1+\varepsilon)$-biLipschitz map
 $\varphi \colon U \to \R^n$ with $\varphi(U)$ open so that for every constant speed geodesic $\gamma \colon [0,1] \to U$ the limit
 \[
  \lim_{t \searrow 0}\frac{\varphi(\gamma(t))-\varphi(\gamma(0))}{d(\gamma(t),\gamma(0))}
 \]
 exists.
\end{theorem}
\begin{proof}
We recall (see \cite{OtsuShioya} or \cite[Theorem 10.8.4]{BuragoBuragoIvanov}) that the local $(1+\varepsilon)$-biLipschitz chart $\varphi \colon U \to \R^n$ can be obtained as
\[
 \varphi(x) = (d(a_1,x),d(a_2,x),\dots,d(a_n,x)),
\]
where $(a_i,b_i)_{i=1}^n$ is a $\delta$-strainer for $p$, for some $\delta > 0$.
Now, the first variation formula (see \cite[Theorem 3.5]{OtsuShioya} or \cite[Theorem 4.5.6, Corollary 4.5.7]{BuragoBuragoIvanov}) implies that 
\[
 \lim_{t \searrow 0} \frac{d(a_i,\gamma(t))-d(a_i,\gamma(0))}{d(\gamma(t),\gamma(0))} = -\cos(\alpha),
\]
where $\alpha = \measuredangle(\gamma,\beta)$, with $\beta$ a geodesic from $\gamma(0)$ to $a_i$. Thus, the required limit exists for each $i$.
% Using the form of the chart we easily conclude the existence of the required limits:
% {\color{red}The following is just the First variation formula... So refer to it and erase the proof.}
% Let $\gamma \colon [0,T] \to U$ be a unit speed geodesic and $i \in \{1, \dots, n\}$. Then, by Lemma \ref{lma:alex_conv} the limit 
% \[
% \alpha :=  \lim_{t \searrow 0} \tilde\measuredangle_k(a_i,\gamma(0),\gamma(t))
% \]
% exists. By scaling the distances, we may assume that our curvature lower bound is $k =-1$. Then, by
% the hyperbolic law of cosines in $M_k$, we have
% \begin{align*}
% \lim_{t \searrow 0} &\frac{d(a_i,\gamma(t))-d(a_i,\gamma(0))}{d(\gamma(t),\gamma(0))}
%  =\lim_{t \searrow 0}\frac{\sinh(d(a_i,\gamma(t))-d(a_i,\gamma(0)))}{\sinh(d(\gamma(t),\gamma(0)))}\\
% & = \lim \frac{\cosh(d(a_i,\gamma(0))\left(\cosh(d(\gamma(0),\gamma(t))-\cosh(d(a_i,\gamma(t))-d(a_i,\gamma(0)))\right)}{\sinh(d(a_i,\gamma(0)))\sinh(d(\gamma(0),\gamma(t))} - \cos(\alpha)\\
% & = -\cos(\alpha)
% \end{align*}
% giving the claim for the i-th coordinate, and hence for the whole chart $\varphi$.
\end{proof}

% \begin{theorem}{}
%  For every $p \in \Reg(X)$ and every $\epsilon > 0$ there exists a neighborhood of $p$ that is $(1+\epsilon)$-biLipschitz to an open set in $\R^n$.
% \end{theorem}

\subsection{Optimal mass transportation}

In this section we recall a few basic things in optimal mass transportation.

The Monge-Kantorovich formulation of optimal mass transportation problem (with quadratic cost) is  to investigate for two Borel probability measures $\mu_0$ and $\mu_1$ the following infimum
\begin{equation}\inf \int_{X\times X} d^2(x,y)\,\d\pi(x,y),\end{equation}
where the infimum is taken over all Borel probability measures $\pi\in \mathcal P(X\times X)$ which has $\mu_0$ and $\mu_1$ as a marginals, that is, $\pi(A\times X)=\mu_0(A)$ and $\pi(X\times A)=\mu_1(A)$ for all Borel sets $A\in \B(X)$. In order to guarantee that the above infimum is finite, it is standard to assume the measures $\mu_0$ and $\mu_1$ to have finite second moments. The set of all Borel probability measures in $X$ with finite second moments is denoted by $\mathcal P_2(X)$.

An admissible measure that minimizes the above infimum is called an optimal (transport) plan, and the set of optimal plans between $\mu_0$ and $\mu_1$ is denoted by $\opt(\mu_0,\mu_1)$. We say that an optimal plan $\pi$ is induced by a map, if there exists a Borel measurable function $T\colon X\to X$ so that $\pi =(\id\times T)_\# \mu_0$. Such a map is called an optimal (transport) map. While optimal plans exist under fairly general assumptions \cite{Villani}, the existence of optimal maps is not true in general.

Optimality of a given transport plan depends only on the $c$-cyclical monotonicity of the support of the plan. 
Let us recall this notion.
\begin{definition}[cyclical monotonicity]A set $\Gamma\subset X\times X$ is called c-cyclically monotone, if for all finite sets of points $\{(x_i,y_i)\}_{i=1}^N\subset \Gamma$ the inequality
\begin{equation}\label{cyclicalmonotonicity}\sum_{i=1}^Nd^2(x_i,y_i)\le \sum_{i=1}^Nd^2(x_{\sigma(i)},y_{i})\end{equation}
holds for all permutations $\sigma\in S_N$ of $\{1,\dots,N\}$.
\end{definition}
%
%A basic characterization of optimality is the c-cyclically monotonicity of the support.
A characterization of optimality using $c$-cyclical monotonicity of the support that is sufficient for us is  the following result proven in \cite{Pratelli} which holds for continuous cost functions.
\begin{lemma}[{\cite[Theorem B]{Pratelli}}]\label{lma:cyclical}
Let $X$ be a Polish space and $\mu_0, \mu_1\in\mathcal{P}_2(X)$. Then a transport plan $\pi$ between $\mu_0$ and $\mu_1$ is optimal if and only if its support is c-cyclically monotone set.
\end{lemma}

In the following lemma we recall a well-known fact which allows us to localize the problem. One way to prove this is to use the result of Lisini in \cite{Lisini} about Wasserstein geodesics and their lifts to the space of probability measures on geodesics of $X$, see \cite{Galaz-GarciaKellMondinoSosa} for the proof.

\begin{lemma}\label{lma:localcyclical}Let $(X,d)$ be a complete and separable geodesic metric space, and let $\Gamma\subset X\times X$ be a c-cyclically monotone set. Then, the set 
\[
\Gamma_t\coloneqq \{(\gamma(0),\gamma(t))\in X\times X: \gamma\in\geo(X) \mathrm{\ with \ }(\gamma(0),\gamma(1))\in \Gamma\}
\] 
is $c$-cyclically monotone for all $t\in[0,1]$.
\end{lemma}

In order to arrive at a contradiction with cyclical monotonicity, we will use the following lemma.

\begin{lemma}\label{lma:cyclicalcontradicition}
For each $C>1$ there exists $\delta>0$ so that 
\begin{equation}
\frac12|y_1+y_2|^2<(1-\delta)(|y_1|^2+|y_2|^2)
\end{equation}
for all 
\[
y_1,y_2\in K\coloneqq \left\{(y_1,y_2)\in\R^{2n}\,:\,|y_2|=1\textrm{ and } |y_2-y_1|\in\left[\frac1C,C\right]\right\}.
\]
\end{lemma}
\begin{proof}
Let us first observe that for $y_1,y_2\in\R^n$, with $y_1 \ne y_2$ we have
\[
0 < |y_1-y_2|^2 = |y_1|^2 - 2\langle y_1,y_2 \rangle + |y_2|^2
\]
and thus
\begin{equation}\label{eq:nonquantitative}
|y_1+y_2|^2 = |y_1|^2 + 2\langle y_1,y_2 \rangle + |y_2|^2 < 2(|y_1|^2+|y_2|^2).
\end{equation}
The quantitative claim then follows by compactness of $K$: first of all notice that $K\subset \bar{B}(0,2+C)$ and thus $K$ is bounded. The set $K$ is also closed and hence it is compact. The function \[(y_1,y_2)\mapsto \frac{|y_1+y_2|^2}{|y_1|^2+|y_2|^2}\]
is continuous as a function $K\to \R$. Therefore, the maximum of the above function is achieved in $K$. By \eqref{eq:nonquantitative}, this maximum is strictly less than two and hence there exists $\delta>0$ as in the claim.  
\end{proof}

% \begin{lemma}\label{lma:cyclicalcontradicition}
% Let $y_1,y_2\in \R^n$, $y_1\neq y_2$, and $t>0$. Then
% \begin{equation}|y_1|^2+|t(y_2-y_1)-y_2|^2<|y_2|^2+|t(y_2-y_1)-y_1|^2.
% \end{equation}
% %
% Moreover, for each $C>1$ there exists $\delta>0$ so that 
% \begin{equation}2|y_1|^2<(1-\delta)(|y_2|^2+|y_2-2y_1|^2)
% \end{equation}
% for all 
% \[
% y_1,y_2\in K\coloneqq \left\{(y_1,y_2)\in\R^{2n}\,:\,|y_2|=1\textrm{ and } |y_2-y_1|\in\left[\frac1C,C\right]\right\}
% \]
% \end{lemma}
% %
% \begin{proof}
% Let $y_1,y_2\in\R^n$ and $t>0$. Denote $w\coloneqq t(y_2-y_1)$. Then $\langle w,y_2-y_1\rangle=|w||y_2-y_1|=t|y_2-y_1|^2$. Thus
% \begin{align}|y_1|^2+|w-y_2|^2&=|y_2-y_2+y_1|^2+|w-y_1+y_1-y_2|^2\\
% &=|y_2|^2+|w-y_1|^2+2|y_2-y_1|^2+2\langle y_2,y_1-y_2\rangle+2(w-y_1,y_1-y_2\rangle\\
% &=|y_2|^2+|w-y_1|^2+2|y_2-y_1|^2+2\langle y_2+(w-y_1),y_1-y_2\rangle\\
% &=|y_2|^2+|w-y_1|^2+2\langle w,y_1-y_2\rangle\\
% &=|y_2|^2+|w-y_1|^2-2t|y_2-y_1|^2\\
% &<|y_2|^2+|w-y_1|^2.
% \end{align}

% The second part of the claim follows by compactness of $K$: first of all notice that $K\subset B(0,1+C)$ and thus bounded. The set $K$ is also closed and hence it is compact. The function \[(y_1,y_2)\mapsto \frac{2|y_1|^2}{|y_2|^2+|(y_2-y_1)-y_1|^2}\]
% is continuous as a function $K\to \R$. Therefore, the maximum of the above function is achieved in $K$. By the first part the lemma, this maximum is strictly less than 1 and hence there exists $\delta>0$ as in the claim.  
% \end{proof}

\section{Proof of Theorem \ref{thm:main}}\label{sec:proof}

In order to prove the uniqueness of optimal transport plans it suffices to show that any optimal transport plan is induced by a map. Indeed, if there were two different optimal plans $\pi_1$ and $\pi_2$, then their convex combination $\frac12(\pi_1+\pi_2)$ would also be optimal and not given by a map.
We will prove Theorem \ref{thm:main} by assuming that there exists an optimal plan that is not induced by a map, then localizing to a chart and using an Euclidean argument to find a contradiction.

\noindent
\textbf{Step 1: initial uniform bounds and measurable selections}\\
Let $\mu_0, \mu_1 \in \mathcal P_2(X)$ with $\mu_0$ purely $(n-1)$-unrectifiable. Let $\pi$ be an optimal  plan from $\mu_0$ to $\mu_1$. Towards a contradiction, we assume that $\pi$ is not induced by a map, that is, there does not exist a Borel map $T \colon X \to X$ so that $\pi = ({\rm id},T)_\sharp\mu_0$. 
%%%%%%
Consider the set 
\[A\coloneqq\{x\in X:\text{ there exist } y^1,y^2\in X \textrm{\ such that\ } (x,y^1),(x,y^2)\in\spt(\pi)\}.  \]
Since $A$ is a projection of a Borel set 
\[\{(x,y,z,w)\in \spt(\pi)\times\spt(\pi): d(x,z)=0,\ d(y,w)>0\},\]
it is a Souslin set and thus $\mu_0$-measurable. (Actually, as a projection of a $\sigma$-compact set, $A$ is Borel.) We will show that $A$ has positive $\mu_0$ measure.

For that we will first show that there exists a Borel selection $T\colon \mathtt{p}_1(\spt(\pi)) \to X $ of $\spt(\pi)$, where $\mathtt{p}_1\colon X\times X \to X$ is the projection to the first coordinate. Define
\[(\spt(\pi))_x\coloneqq\{y\in X: (x,y)\in\spt(\pi)\}.\]
Then $(\spt(\pi))_x=(\{x\}\times X)\cap\spt(\pi)$ and thus it is closed. Furthermore, as a proper space, $X$ is also $\sigma$-compact, and thus so is $(\spt(\pi))_x$. Hence, by the Arsenin-Kunugui Theorem \cite[Theorem 35.46]{Kechris} there exists a Borel selection of $\spt(\pi)$, in other words, there exists a Borel map $T\colon \mathtt{p}_1(\spt(\pi))\to X$ with $\mathtt{p}_1(\spt(\pi))$ Borel so that $T(x)\in (\spt(\pi))_x$ for all $x\in \mathtt{p}_1(\spt(\pi))$. 

Suppose now that $\mu_0(A)=0$. We will show that in this case $\pi$ would be induced by the map $T$. Indeed, for $E\subset X\times X$ we have that
\begin{align}(\id,T)_\#\mu_0(E)&=\mu_0((\id,T)^{-1}(E)\setminus A)=\mu_0(\mathtt{p}_1(E\cap \mathrm{Graph}(T)))\\&=\pi((\mathtt{p}_1(E\cap\mathrm{Graph}(T))\times X)\cap\spt(\pi))
\\&=\pi((E\cap\spt(\pi))\setminus(A\times X))=\pi(E\cap\spt(\pi))=\pi(E).
\end{align}
Thus $\mu_0(A)>0$.

Since $X$ is geodesic, for all $x\in A$ there exist $\gamma^1_x,\gamma^2_x\in\geo(X)$ such that $\gamma^1_x(0)=x=\gamma^2_x(0)$, $\gamma^1_x(1)\neq \gamma^2_x(1)$, and $(\gamma^i_x(0),\gamma^i_x(1))\in\spt (\pi)$ for $i\in\{1,2\}$. We will need to choose the geodesics $\gamma_x^1$ and $\gamma_x^2$ in a measurable way. We will also make the selection so that
\begin{align}\label{extracond}
d(x,\gamma_x^1(1)) \le d(x,\gamma_x^2(1)) \ne 0.
\end{align}
%%%%%
%%%%%
By now, we have a Borel selection $T$ of $\spt(\pi)$. Since $\mathtt{p}_1(\spt(\pi))$ is a Borel set, we can extend $T$ to a Borel map $T\colon X\to X$. Consider now the set $\spt(\pi)\setminus \mathrm{Graph(T)}$. Since $T$ is a Borel map, the graph of $T$ is a Borel set and thus the set $\spt(\pi)\setminus \mathrm{Graph(T)}$ is a Borel set. Since $X\setminus T(x)$ is $\sigma$-compact by the properness and separability of $X$, we have that $(\spt(\pi)\setminus \mathrm{Graph(T)})_x$ is $\sigma$-compact as a closed subset of $X\setminus T(x)$. Thus again by the Arsenin-Kunugui Theorem 
%\cite[Theorem 35.46]{Kechris} 
there exists a Borel selection $S\colon \mathtt{p}_1(\spt(\pi)\setminus \mathrm{Graph(T)})\to X$ that we can further extend to a Borel map $S\colon X\to X$ for which we have that $T(x)=S(x)$ for $x\notin A$, and $T(x)\neq S(x)$ for $x\in A$. 
%{\color{red} This actually shows that A is a Borel set}

To have \eqref{extracond} we will define two auxiliary maps $\tilde T^1,\tilde T^2\colon X\to X\times X$ as
\[\tilde T^1(x)\coloneqq \left\{\begin{array}{cc}(x,T(x)),& x\in h^{-1}(-\infty,0)\\(x,S(x)),& x\in h^{-1}[0,\infty),\end{array}\right.\]
where $h(x)\coloneqq d(x,T(x))-d(x,S(x))$, and similarly
\[\tilde T^2(x)\coloneqq \left\{\begin{array}{cc}(x,S(x)),& x\in h^{-1}(-\infty,0)\\(x,T(x)),& x\in h^{-1}[0,\infty).\end{array}\right.\]
The maps $\tilde T^1$ and $\tilde T^2$ are Borel maps since $T,S$ and $h$ are Borel maps.

It remains to select the geodesics between points $x$ and $T^i(x)$. For that, we consider the set 
\[G\coloneqq \{(x,y,\gamma)\in X\times X\times \geo(X)\,:\, \gamma(0)=x,\gamma(1)=y\}.
\]
The set $G$ is Borel as the preimage of zero under the Borel map 
\[
(x,y,\gamma)\mapsto \sup\{d(x,\gamma(0)),d(y,\gamma(1))\}.
\]
Furthermore, we have by the Arzel\`a-Ascoli Theorem that 
\[
G_{(x,y)} \coloneqq \{\gamma\in \geo(X) \,:\, \gamma(0)=x,\gamma(1)=y\}.
\]
is compact. Thus, by the Arsenin-Kunugui Theorem there exists a Borel selection $F\colon X\times X\to G_{(x,y)}$. With this we may finally define $T^1,T^2\colon X\to \geo(X)$ as 
\begin{align}T^1&\coloneqq F\circ \tilde T^1\quad\mathrm{and}
\\ T^2&\coloneqq F\circ \tilde T^2.\end{align}
%%%%%
%%%%%
From now on, we will denote $\gamma^1_x=T^1(x)$ and $\gamma^2_x=T^2(x)$ for all $x\in A$. Notice that $\gamma_x^1$ and $\gamma_x^2$ satisfy \eqref{extracond}.

By Lemma \ref{lma:positive_distance_limit}, we have for all $x \in A$ that
\[\lim_{t\searrow 0}\frac{d(\gamma^1_x(t),\gamma^2_x(t))}{d(x,\gamma_x^2(t))}\in(0,\infty).\]
Thus, we may write $A$ as a countable union of sets 
\begin{align*}
A_i\coloneqq\bigg\{x\in A \,:\,  d(x,\gamma^2_x(1))\in\left[1/i,i\right]\text{ and }
&\frac{d(\gamma^1_x(t),\gamma^2_x(t))}{d(x,\gamma^2(t))}\in\left[1/i,i\right]\text{ for all } t\le\frac{1}i\bigg\},
\end{align*}
and therefore there exists $k\in\N$ so that $\mu_0(A_k)>0$. Notice that the sets $A_i$ are measurable, since we can write $A_i$ as the intersection of
\[
\left\{x\in A: d(x,\gamma_x^2(1))\in[1/i,i]\right\}
\]
and
\begin{align}
%\left\{x\in X: \frac{d(\gamma_x^1(t),\gamma_x^2(t))}{d(x,\gamma_x^2(t))}\in[1/i,i]\text{ for all }t\le \frac1i\right\} &
%= 
\bigcap_{\substack{t\le\frac1i\\t \in \mathbb Q}}\left\{x\in X:\frac{d(\gamma_x^1(t),\gamma_x^2(t))}{d(x,\gamma_x^2(t))}\in[1/i,i]\right\}.
\end{align}
We now consider $k \in \N$ fixed so that $\mu_0(A_k) > 0$.

\noindent
\textbf{Step 2: localization to a chart}\\
Now we are ready to localize the problem so that we may use properties of the Euclidean space to arrive to the contradiction. 
We will need to choose $\varepsilon>0$ sufficiently small to arrive to a contradiction with c-cyclical monotonicity in a $(1+\varepsilon)$-chart given by Theorem \ref{thm:epsilonneighborhood}. 
We define
\begin{equation}\label{eq:epsilonchoice}
\varepsilon := \frac{\delta}{100} \in (0,1/200),
\end{equation}
where $\delta=\delta(2k) \in (0,1/2)$ is the constant given by Lemma \ref{lma:cyclicalcontradicition} for the $k$ fixed above.
%
%We will not specify the constant $\varepsilon$, but instead point out that it only depends on the constant $\delta=\delta(k)$ given by Lemma \ref{lma:cyclicalcontradicition} with $k$ fixed above. We will also use an auxiliary constant $\hat\varepsilon>0$ depending on $k$ and $\varepsilon$ which we will not specify either.
%
Since $\mu_0$ is purely $(n-1)$-unrectifiable and $\Sing(X)$ is $(n-1)$-rectifiable by Theorem \ref{thm:singsize}, we have $\mu_0(A_k\cap \Reg(X))= \mu_0(A_k)$. By Theorem \ref{thm:epsilonneighborhood} we can cover the set $\Reg(X)$ with open sets $U$ for which the associated maps $\varphi\colon U\to \R^n$ are $(1+\varepsilon)$-biLipschitz, and the limit  
\[
  \lim_{t \searrow 0}\frac{\varphi(\gamma(t))-\varphi(\gamma(0))}{d(\gamma(t),\gamma(0))}
\]
exists for all geodesics $\gamma\subset U$. Since $X$ is a proper metric space, it is in particular hereditarily Lindel\"of. Therefore, there exists a countable subcover $\mathcal{F}$ of such open sets $U$. Hence, there exists $U\in\F$ for which $\mu_0(U\cap A_k)>0$. Let $\varphi\colon U\to\R^n$ be as in Theorem \ref{thm:epsilonneighborhood}.

\noindent
\textbf{Step 3: discretization and choice of points for the contradiction}\\
Next we take a subset of $A_k\cap U$ where the direction of the two selected geodesics is independent of the point, up to a small error 
\begin{equation}\label{eq:epsilonhatchoice}
\hat\varepsilon\coloneqq \frac{\varepsilon}{80k^4}>0.
\end{equation}
This is done by covering the set $\R^n$ by sets $\{B(y_i,\hat\varepsilon)\}_{i\in\N}$.
%, where $\{y_i\}\subset \varphi(U)$ is a dense subset of $\varphi(U)$. 
Then there exist $i$, $j$ and $t_0>0$ so that the set 
\begin{align*}
B\coloneqq\bigg\{x\in A_k\cap U\,:\, & \frac{\varphi(\gamma^1_x(t))-\varphi(x)}t\in B(y_i,\hat\varepsilon), \frac{\varphi(\gamma^2_x(t))-\varphi(x)}t\in B(y_j,\hat\varepsilon),\\
& \varphi(\gamma^1_x(t)), \varphi(\gamma^2_x(t)) \in U\text{ for all } t\le t_0\bigg\}
\end{align*}
has positive $\mu_0$-measure. Notice that $B$ is seen to be measurable by a similar argument than $A_i$. By relabeling, we may assume that $i=1$ and $j=2$.

Since $\varphi$ is biLipschitz, the measure $\varphi_\#\mu_0$ is purely $(n-1)$-unrectifiable on $\R^n$. Hence, by Lemma \ref{lma:unrecteverywhere} there exist points $x_1,x_2\in B$ such that
\begin{equation}\label{eq:cone}
\varphi(x_2)\in C\left(\varphi(x_1),\frac{y_2-y_1}{|y_2-y_1|}, \hat\varepsilon\right)\cap B(\varphi(x_1),r),
\end{equation}
where $r\le \hat\varepsilon$ is such that $r\le \frac{t_0}{2}|y_2-y_1|$. 
Now that we have selected the initial points $x_1$ and $x_2$ for the contradiction argument, we still need to
bring the target points close enough to $x_1$ and $x_2$ by contracting along the geodesics $\gamma_{x_1}^2$ and $\gamma_{x_2}^1$.
Since $|\varphi(x_1)-\varphi(x_2)|<r$, there exists the desired contraction parameter $t\le t_0$ for which 
\begin{equation}\label{eq:tchoice}
2|\varphi(x_2)-\varphi(x_1)|= |ty_2-ty_1|.
\end{equation}
We will now use as target points the points $\gamma_{x_1}^2(t)$ and $\gamma_{x_2}^1(t)$.

\noindent
\textbf{Step 4: verifying the bounds for Lemma \ref{lma:cyclicalcontradicition}}\\
In the remainder of the proof we verify that the four selected points $x_2,x_1,\gamma_{x_1}^2(t)$ and $\gamma_{x_2}^1(t)$ give a contradiction with $c$-cyclical monotonicity.
Towards this goal we first check that we may apply Lemma \ref{lma:cyclicalcontradicition} with the selected $\delta$.

First of all, we have by the definition of $A_k$ that
\begin{align}
\frac{|\varphi(\gamma_{x_1}^2(t))-\varphi(\gamma_{x_1}^1(t))|}{|\varphi(\gamma_{x_1}^2(t))-\varphi({x_1})|}\in \left[\frac1{(1+\varepsilon)^2k},(1+\varepsilon)^2k\right].
\end{align}
Since 
\[
\hat\varepsilon\le \frac{\varepsilon}{2(1+\varepsilon)k^2},
\]
we have by the fact that $x_1 \in A_k$ and $\varphi$ is $(1+\varepsilon)$-biLipschitz, that 
\begin{align}2t\hat\varepsilon\le \frac{\varepsilon}{(1+\varepsilon)}\frac{d(\gamma_{x_1}^2(t),x_1)}{k}\le\frac{\varepsilon d(\gamma_{x_1}^2(t),\gamma_{x_1}^1(t))}{(1+\varepsilon)}\le\varepsilon|\varphi(\gamma_{x_1}^2(t))-\varphi(\gamma_{x_1}^1(t))|.\end{align}
Similarly, since $\hat\varepsilon\le \frac{\varepsilon}{(1+\varepsilon)k}$, we have that
\[t\hat\varepsilon\le\varepsilon|\varphi(\gamma_{x_1}^2(t))-\varphi(x_1)|.\]
Therefore, we have by the fact that $x_1 \in B$, the triangle inequality and the choice of $\varepsilon$ and $\hat\varepsilon$ that 
\begin{align}
\frac{|ty_2-ty_1|}{|ty_2|}&\le\frac{|\varphi(\gamma_{x_1}^2(t))-\varphi(\gamma_{x_1}^1(t))|+2t\hat\varepsilon}{|\varphi(\gamma_{x_1}^2(t))-\varphi(x_1)|-t\hat\varepsilon}\\
&\le\frac{(1+\varepsilon)}{(1-\varepsilon)}\frac{|\varphi(\gamma_{x_1}^2(t))-\varphi(\gamma_{x_1}^1(t))|}{|\varphi(\gamma_{x_1}^2(t))-\varphi(x_1)|}\\
&\le\frac{(1+\varepsilon)}{(1-\varepsilon)}(1+\varepsilon)^2k< 2k
\end{align}
By similar arguments, we have that
\[\frac{|ty_2-ty_1|}{|ty_2|}>\frac{1}{2k}.\]
Thus, by Lemma \ref{lma:cyclicalcontradicition} with the $\delta=\delta(2k)$ already chosen accordingly, we have
\begin{align}\frac{\frac12|t(y_1+y_2)|^2}{|ty_2|^2}<(1-\delta)\frac{(|ty_2|^2+|ty_1|^2)}{|ty_2|^2},
\end{align}
that is,
\begin{align}
\label{cyc}\frac12|t(y_1+y_2)|^2<(1-\delta)(|ty_2|^2+|ty_1|^2).
\end{align}

\noindent
\textbf{Step 5: the contradiction}\\
We will then use the inequality \eqref{cyc} to get to a contradiction with the c-cyclical monotonicity guaranteed by Lemma \ref{lma:localcyclical}. Let us first estimate the terms on the right-hand side of \eqref{cyc}.

By the definition of $y_1$ and $A_k$ we have that
\begin{align}|ty_1|&\le|ty_1-\varphi(\gamma_{x_1}^1(t))+\varphi(x_1)|+|\varphi(\gamma_{x_1}^1(t))-\varphi(x_1)|\\
&\le t\hat\varepsilon+(1+\varepsilon)d(\gamma_{x_1}(t),x_1)\le tk+(1+\varepsilon)tk\le 3tk.\end{align}
Similarly, \[|ty_2|\le 3tk.\]
Therefore, we have that
\begin{align}\label{sumdifestim}|\frac12t(y_1+y_2)|,|\frac12t(y_2-y_1)|\le 3tk.\end{align}

Using the definition of the set $B$, and \eqref{eq:tchoice}, \eqref{eq:cone} and \eqref{sumdifestim}, we have
\begin{align}
\frac1{(1+\varepsilon)^2}&d^2(x_2,\gamma^2_{x_1}(t)) \le |\varphi(\gamma^2_{x_1}(t))-\varphi(x_2)|^2\\
&=|\frac12t(y_1+y_2)+(\varphi(\gamma^2_{x_1}(t))-\varphi(x_1)-ty_2)
-(\varphi(x_2)-\varphi(x_1)-\frac12t(y_2-y_1))|^2\\
&\le(|\frac12t(y_1+y_2)|+|\varphi(\gamma^2_{x_1}(t))-\varphi(x_1)-ty_2|
+|\varphi(x_2)-\varphi(x_1)-\frac12t(y_2-y_1)|)^2\\
&\le (|\frac12t(y_1+y_2)|+t\hat\varepsilon+\frac12|t(y_2-y_1)|\hat\varepsilon)^2 
\le (|\frac12t(y_1+y_2)|+(3k+1)t\hat\varepsilon)^2 \\
& \le |\frac12t(y_1+y_2)|^2 + 6tk (3k+1)t\hat\varepsilon + ((3k+1)t\hat\varepsilon)^2
\le |\frac12t(y_1+y_2)|^2 + 40 t^2k^2\hat\varepsilon
\end{align}
and similarly
\begin{align}
\frac1{(1+\varepsilon)^2}d^2(x_1,\gamma^2_{x_2}(t))\le |\frac12t(y_1+y_2)|^2 + 40 t^2k^2\hat\varepsilon.
\end{align}

Thus, by summing the two terms, using \eqref{eq:epsilonhatchoice} and the fact that $x_1 \in A_k$,
\begin{equation}\label{error}
\begin{split}
\frac1{(1+\varepsilon)^2}&[d^2(x_2,\gamma^2_{x_1}(t))+d^2(x_1,\gamma^1_{x_2}(t))]
 \le 2|\frac12t(y_1+y_2)|^2+80 t^2k^2\hat\varepsilon \\
& \le \frac12|t(y_1+y_2)|^2 + \frac{t^2}{k^2}\varepsilon \le \frac12|t(y_1+y_2)|^2 + \varepsilon d^2(\gamma^2_{x_1}(t),x_1),
\end{split}
\end{equation}
% where
% \begin{align}
% E= & 2t^2\hat\varepsilon^2+2t^2|\frac12(y_2-y_1)|^2\hat\varepsilon^2+4|\frac12t(y_1+y_2)|t\hat\varepsilon\\
% & +4|\frac12t(y_1+y_2)||t(y_2-y_1)|\hat\varepsilon+4t^2\hat\varepsilon^2|\frac12t(y_2-y_1)|.
% \end{align}
%
%

% Hence, by the choice of $\hat\varepsilon$ we have for the error term $E$ that 
% \begin{align}\label{error}
% E\le 200t^2\hat\varepsilon k^2\le\varepsilon d^2(\gamma^2_{x_1}(t),x_1),
% \end{align}
% due to the fact that $x_1\in A_k$.
%
%
%Since $\hat\varepsilon< \frac{(1+\varepsilon)k^2}2$, by the definition of $A_k$ we have that
%\begin{align}|t(y_2-ty_1)|& \le |ty_2-\varphi(\gamma^2_{x_1}(t))+\varphi(x_1)|+|\varphi(\gamma^1_{x_1}(t))-\varphi(x_1)-ty_1|+|\varphi(\gamma^2_{x_1}(t))-\varphi(\gamma^1_{x_1}(t))|\\
%& \le 2t\hat\varepsilon+(1+\varepsilon)kd(x_1,\gamma^2_{x_1})\le 2t\hat\varepsilon+(1+\varepsilon)tk^2\le 2(1+\varepsilon)tk^2. 
%\end{align}
%Furthermore, since $\hat\varepsilon<k$, we have that
%\begin{align}|\frac12t(y_1+y_2)|\le t\hat\varepsilon+tk\le2tk.
%\end{align}
%Hence, we have for the error term $E$ that 
%\begin{align}\label{error}
%E\le t^2\hat\varepsilon C(k,\varepsilon),
%\end{align}
%where $C(k,\varepsilon)>0$ is a constant only depending on $k$ and $\varepsilon$. 
%Thus, by choosing $\hat\varepsilon\le \frac{\varepsilon}{k^2C(k,\varepsilon)}$ we have, due to the fact that $x_1 \in A_k$, that
%\begin{align}E\le\varepsilon d^2(\gamma^2_{x_1}(t),x_1).
%\end{align}
Again,  by the definition of the set $B$ and the choice of $\hat\varepsilon$ 
\begin{align}
|ty_1|^2& \le \left((1+\varepsilon)d(\gamma^1_{x_2}(t),x_2)+t\hat\varepsilon\right)^2\le \left((1+\varepsilon)d(\gamma^1_{x_2}(t),x_2) + \varepsilon d(\gamma^2_{x_1}(t),x_1)\right)^2 \\
& \le((1+\varepsilon)^2+2(1+\varepsilon)\varepsilon)d^2(\gamma^1_{x_2}(t),x_2) 
+(\varepsilon^2 +2(1+\varepsilon)\varepsilon) d^2(\gamma^2_{x_1}(t),x_1)\label{tyestim1}\\
& \le (1 + 7\varepsilon) d^2(\gamma^1_{x_2}(t),x_2) + 5\varepsilon d^2(\gamma^2_{x_1}(t),x_1)
\end{align}
and
\begin{equation}
\begin{split}
\label{tyestim2}
|ty_2|^2 &\le \left((1+\varepsilon)d(\gamma^2_{x_1}(t),x_1)+t\hat\varepsilon\right)^2\\
& \le (1+2\varepsilon)^2d^2(\gamma^2_{x_1}(t),x_1) \le (1+8\varepsilon) d^2(\gamma^2_{x_1}(t),x_1).
\end{split}
\end{equation}
Using the inequalities \eqref{cyc}, \eqref{tyestim1} and \eqref{tyestim2}, we get that
\begin{equation}
\begin{split}
\label{sumestim}
\frac12|t(y_1+y_2)|^2 & <(1-\delta)(|ty_2|^2+|ty_1|^2)\\ 
& \le (1-\delta)(1+13 \varepsilon)(d^2(\gamma^2_{x_1(t)},x_1)+d^2(\gamma^1_{x_2(t)},x_2)).
\end{split}
\end{equation}
Hence, by \eqref{error}, \eqref{sumestim}, the fact that $\delta \le \frac12$ and the choice of $\varepsilon$, we have that 
\begin{align}
d^2(x_2,&\gamma^2_{x_1}(t))+d^2(x_1,\gamma^1_{x_2}(t))\\
& \le (1+\varepsilon)^2\left(\frac12|t(y_1+y_2)|^2 + \varepsilon d^2(\gamma^2_{x_1}(t),x_1)\right)\\
& \le (1+\varepsilon)^2(1-\delta)(1+15\varepsilon)(d^2(\gamma^2_{x_1}(t),x_1)+d^2(\gamma^1_{x_2}(t),x_2))\\
& \le (1-\delta)(1+100\varepsilon)(d^2(\gamma^2_{x_1}(t),x_1)+d^2(\gamma^1_{x_2}(t),x_2))\\
& <d^2(x_2,\gamma^1_{x_2}(t))+d^2(x_1,\gamma^2_{x_1}(t))
\end{align}
However, since $(x_2,\gamma_{x_2}^1(1)),(x_1,\gamma_{x_1}^2(1))\in \spt(\pi)$ we have by Lemma \ref{lma:localcyclical} that 
\begin{align}\label{cycmonotforcontradiction}
d^2(x_2,\gamma^1_{x_2}(t))+d^2(x_1,\gamma^2_{x_1}(t))\le d^2(x_2,\gamma^2_{x_1}(t))+d^2(x_1,\gamma^1_{x_2}(t)).
\end{align}
which is a contradiction.
Therefore, the plan $\pi$ is induced by a map.

\bibliographystyle{amsplain}
\bibliography{OTreference}

\providecommand{\bysame}{\leavevmode\hbox to3em{\hrulefill}\thinspace}
\providecommand{\MR}{\relax\ifhmode\unskip\space\fi MR }
\providecommand{\MRhref}[2]{%
  \href{http://www.ams.org/mathscinet-getitem?mr=#1}{#2}
}
\providecommand{\href}[2]{#2}
\begin{thebibliography}{10}

\bibitem{AmbrosioGigliMondinoRajala}
Luigi Ambrosio, Nicola Gigli, Andrea Mondino, and Tapio Rajala,
  \emph{Riemannian {R}icci curvature lower bounds in metric measure spaces with
  {$\sigma$}-finite measure}, Trans. Amer. Math. Soc. \textbf{367} (2015),
  no.~7, 4661--4701. \MR{3335397}

\bibitem{AmbrosioGigliSavare1}
Luigi Ambrosio, Nicola Gigli, and Giuseppe Savar\'e, \emph{Calculus and heat
  flow in metric measure spaces and applications to spaces with {R}icci bounds
  from below}, Invent. Math. \textbf{195} (2014), no.~2, 289--391. \MR{3152751}

\bibitem{AmbrosioGigliSavare2}
\bysame, \emph{Metric measure spaces with {R}iemannian {R}icci curvature
  bounded from below}, Duke Math. J. \textbf{163} (2014), no.~7, 1405--1490.
  \MR{3205729}

\bibitem{AmbrosioMondinoSavare}
Luigi Ambrosio, Andrea Mondino, and Giuseppe Savar\'e, \emph{Nonlinear
  diffusion equations and curvature conditions in metric measure spaces},
  Memoirs Amer. Math. Soc. (to appear).

\bibitem{AmbrosioRajala}
Luigi Ambrosio and Tapio Rajala, \emph{Slopes of {K}antorovich potentials and
  existence of optimal transport maps in metric measure spaces}, Ann. Mat. Pura
  Appl. (4) \textbf{193} (2014), no.~1, 71--87. \MR{3158838}

\bibitem{Bertrand}
J\'er\^ome Bertrand, \emph{Existence and uniqueness of optimal maps on
  {A}lexandrov spaces}, Adv. Math. \textbf{219} (2008), no.~3, 838--851.
  \MR{2442054}

\bibitem{Bertrand_habi}
\bysame, \emph{Alexandrov, {K}antorovitch et quelques autres. {E}xemples
  d'interactions entre transport optimal et g\'eom\'etrie d'{A}lexandrov},
  Manuscrit pr\'esent\'e pour l'obtention de l'Habilitation \`a Diriger des
  Recherches (2015).

\bibitem{Brenier}
Yann Brenier, \emph{Polar factorization and monotone rearrangement of
  vector-valued functions}, Comm. Pure Appl. Math. \textbf{44} (1991), no.~4,
  375--417. \MR{1100809}

\bibitem{BuragoBuragoIvanov}
Dmitri Burago, Yuri Burago, and Sergei Ivanov, \emph{A course in metric
  geometry}, Graduate Studies in Mathematics, vol.~33, American Mathematical
  Society, Providence, RI, 2001. \MR{1835418}

\bibitem{BuragoGromovPerelman}
Yu. Burago, M.~Gromov, and G.~Perel'man, \emph{A. {D}. {A}leksandrov spaces
  with curvatures bounded below}, Uspekhi Mat. Nauk \textbf{47} (1992),
  no.~2(284), 3--51, 222. \MR{1185284}

\bibitem{CavallettiHuesmann}
Fabio Cavalletti and Martin Huesmann, \emph{Existence and uniqueness of optimal
  transport maps}, Ann. Inst. H. Poincar\'e Anal. Non Lin\'eaire \textbf{32}
  (2015), no.~6, 1367--1377. \MR{3425266}

\bibitem{CavallettiMondino}
Fabio Cavalletti and Andrea Mondino, \emph{Optimal maps in essentially
  non-branching spaces}, Commun. Contemp. Math. \textbf{19} (2017), no.~6,
  1750007, 27. \MR{3691502}

\bibitem{ErbarKuwadaSturm}
Matthias Erbar, Kazumasa Kuwada, and Karl-Theodor Sturm, \emph{On the
  equivalence of the entropic curvature-dimension condition and {B}ochner's
  inequality on metric measure spaces}, Invent. Math. \textbf{201} (2015),
  no.~3, 993--1071. \MR{3385639}

\bibitem{Galaz-GarciaKellMondinoSosa}
Fernando Galaz-Garc\'{i}a, Martin Kell, Andrea Mondino, and Gerardo Sosa,
  \emph{On quotients of spaces with ricci curvature bounded below}, Preprint,
  arXiv:1704.05428 (2017).

\bibitem{GangboMcCann}
Wilfrid Gangbo and Robert~J. McCann, \emph{The geometry of optimal
  transportation}, Acta Math. \textbf{177} (1996), no.~2, 113--161.
  \MR{1440931}

\bibitem{Gigli2}
Nicola Gigli, \emph{On the inverse implication of {B}renier-{M}c{C}ann theorems
  and the structure of {$(\mathcal{P}_2(M),W_2)$}}, Methods Appl. Anal.
  \textbf{18} (2011), no.~2, 127--158. \MR{2847481}

\bibitem{Gigli}
\bysame, \emph{Optimal maps in non branching spaces with {R}icci curvature
  bounded from below}, Geom. Funct. Anal. \textbf{22} (2012), no.~4, 990--999.
  \MR{2984123}

\bibitem{GigliRajalaSturm}
Nicola Gigli, Tapio Rajala, and Karl-Theodor Sturm, \emph{Optimal maps and
  exponentiation on finite-dimensional spaces with {R}icci curvature bounded
  from below}, J. Geom. Anal. \textbf{26} (2016), no.~4, 2914--2929.
  \MR{3544946}

\bibitem{Kantorovitch2}
L.~Kantorovich, \emph{On a problem of monge (in russian)}, Uspekhi Mat. Nauk.
  \textbf{3} (1948), 225--226.

\bibitem{Kantorovitch1}
L.~Kantorovitch, \emph{On the translocation of masses}, C. R. (Doklady) Acad.
  Sci. URSS (N.S.) \textbf{37} (1942), 199--201. \MR{0009619}

\bibitem{Kechris}
Alexander~S. Kechris, \emph{Classical descriptive set theory}, Graduate Texts
  in Mathematics, vol. 156, Springer-Verlag, New York, 1995. \MR{1321597}

\bibitem{Kell}
Martin Kell, \emph{Transport maps, non-branching sets of geodesics and measure
  rigidity}, Adv. Math. \textbf{320} (2017), 520--573. \MR{3709114}

\bibitem{KettererRajala}
Christian Ketterer and Tapio Rajala, \emph{Failure of topological rigidity
  results for the measure contraction property}, Potential Anal. \textbf{42}
  (2015), no.~3, 645--655. \MR{3336992}

\bibitem{Lisini}
Stefano Lisini, \emph{Characterization of absolutely continuous curves in
  {W}asserstein spaces}, Calc. Var. Partial Differential Equations \textbf{28}
  (2007), no.~1, 85--120. \MR{2267755}

\bibitem{LottVillani}
John Lott and C\'edric Villani, \emph{Ricci curvature for metric-measure spaces
  via optimal transport}, Ann. of Math. (2) \textbf{169} (2009), no.~3,
  903--991. \MR{2480619}

\bibitem{McCann}
Robert~J. McCann, \emph{Polar factorization of maps on {R}iemannian manifolds},
  Geom. Funct. Anal. \textbf{11} (2001), no.~3, 589--608. \MR{1844080}

\bibitem{Monge}
G.~Monge, \emph{Mémoire sur la théorie des déblais et remblais}, Histoire de
  l'Académie Royale des Sciences de Paris (1781), 666--704.

\bibitem{Ohta}
Shin-ichi Ohta, \emph{On the measure contraction property of metric measure
  spaces}, Comment. Math. Helv. \textbf{82} (2007), no.~4, 805--828.
  \MR{2341840}

\bibitem{OtsuShioya}
Yukio Otsu and Takashi Shioya, \emph{The {R}iemannian structure of {A}lexandrov
  spaces}, J. Differential Geom. \textbf{39} (1994), no.~3, 629--658.
  \MR{1274133}

\bibitem{Petrunin}
Anton Petrunin, \emph{Alexandrov meets {L}ott-{V}illani-{S}turm}, M\"unster J.
  Math. \textbf{4} (2011), 53--64. \MR{2869253}

\bibitem{Pratelli}
A.~Pratelli, \emph{On the sufficiency of {$c$}-cyclical monotonicity for
  optimality of transport plans}, Math. Z. \textbf{258} (2008), no.~3,
  677--690. \MR{2369050}

\bibitem{RajalaSturm}
Tapio Rajala and Karl-Theodor Sturm, \emph{Non-branching geodesics and optimal
  maps in strong {$CD(K,\infty)$}-spaces}, Calc. Var. Partial Differential
  Equations \textbf{50} (2014), no.~3-4, 831--846. \MR{3216835}

\bibitem{Schultz}
Timo Schultz, \emph{Existence of optimal transport maps in very strict
  {$CD(K,\infty)$} -spaces}, Preprint, arXiv:1712.03670 (2017).

\bibitem{Shiohama}
Katsuhiro Shiohama, \emph{An introduction to the geometry of {A}lexandrov
  spaces}, Lecture Notes Series, vol.~8, Seoul National University, Research
  Institute of Mathematics, Global Analysis Research Center, Seoul, 1993.
  \MR{1320267}

\bibitem{SmithKnott}
C.~S. Smith and M.~Knott, \emph{Note on the optimal transportation of
  distributions}, J. Optim. Theory Appl. \textbf{52} (1987), no.~2, 323--329.
  \MR{879207}

\bibitem{SturmII}
Karl-Theodor Sturm, \emph{On the geometry of metric measure spaces. {I}}, Acta
  Math. \textbf{196} (2006), no.~1, 65--131. \MR{2237206}

\bibitem{SturmI}
\bysame, \emph{On the geometry of metric measure spaces. {II}}, Acta Math.
  \textbf{196} (2006), no.~1, 133--177. \MR{2237207}

\bibitem{Villani}
C\'edric Villani, \emph{Optimal transport}, Grundlehren der Mathematischen
  Wissenschaften [Fundamental Principles of Mathematical Sciences], vol. 338,
  Springer-Verlag, Berlin, 2009, Old and new. \MR{2459454}

\bibitem{zajicek}
Lud\v{e}k Zaj\'i\v{c}ek, \emph{On the differentiation of convex functions in
  finite and infinite dimensional spaces}, Czechoslovak Math. J.
  \textbf{29(104)} (1979), no.~3, 340--348. \MR{536060}

\end{thebibliography}
\end{document}